\documentclass{amsart}
\usepackage{graphicx,pstricks,pst-plot}
\usepackage{hyperref, color}
\usepackage{amsfonts}
\usepackage[portuges,english]{babel}
\usepackage{graphicx}
\usepackage{amscd,color}
\usepackage{amsmath}
\usepackage{amssymb}

\makeatletter
\@namedef{subjclassname@2010}{%
	\textup{2010} Mathematics Subject Classification}
\makeatother

\theoremstyle{plain}

\newtheorem{corollary}{\bf Corollary}

\newtheorem{lemma}{\bf Lemma}

\newtheorem{theorem}{\bf Theorem}
\numberwithin{equation}{section}

\begin{document}
	
	\title[Rigidity and nonexistence of spacelike hypersurfaces]{Rigidity and nonexistence of complete spacelike hypersurfaces in the steady state space}
	
	\author[W.F.C. Barboza, H.F. de Lima, M.A.L. Vel\'{a}squez]{Weiller F.C. Barboza$^1$, Henrique F. de Lima$^{1,\ast}$\\ and Marco Antonio L. Vel\'{a}squez$^1$}
	
	\address{$^1$ Departamento de Matem\'atica, Universidade Federal de Campina Grande, 58.429-970 Campina Grande, Para\'iba, Brazil.}
	\email{weiller@mat.ufcg.edu.br}
	\email{henrique@mat.ufcg.edu.br}
	\email{marco.velasquez@mat.ufcg.edu.br}
	
	
	
	\thanks{$^\ast$ Corresponding author.}
	
	\begin{abstract}
		We study complete spacelike hypersurfaces immersed in an open region of the de Sitter space $\mathbb{S}^{n+1}_{1}$ which is known as the {\em steady state space} $\mathcal{H}^{n+1}$. In this setting, under suitable constraints on the behavior of the higher order mean curvatures of these hypersurfaces, we prove that they must be spacelike hyperplanes of $\mathcal{H}^{n+1}$. Nonexistence results concerning these spacelike hypersurfaces are also given. Our approach is based on a suitable extension of the Omori-Yau's generalized maximum principle due to Al\'{\i}as, Impera and Rigoli in~\cite{Alias:12}.
	\end{abstract}
	
	\maketitle
	
	\section{Introduction}
	
	Let $\mathbb{L}^{n+2}$ denote the $(n+2)$-dimensional Lorentz-Minkowski space ($n\geq2$), that is, the real vector space $\mathbb{R}^{n+2}$ endowed with the
	Lorentz metric defined by
	$$\left\langle v,w\right\rangle ={\displaystyle\sum\limits_{i=1}^{n+1}}v_{i}w_{i}-v_{n+2}w_{n+2},$$
	for all $v,w\in\mathbb{R}^{n+2}$. We define the $(n+1)$-dimensional de Sitter space $\mathbb{S}_{1}^{n+1}$ as the following hyperquadric of $\mathbb{L}^{n+2}$
	$$\mathbb{S}_{1}^{n+1}=\left\{p\in\mathbb L^{n+2}:\left\langle p,p\right\rangle=1\right\}.$$
	The induced metric from $\left\langle ,\right\rangle $ makes $\mathbb{S}_{1}^{n+1}$ a Lorentz manifold with constant sectional curvature one.
	Moreover, for all $p\in\mathbb{S}_{1}^{n+1}$, we have
	$$T_{p}\left(\mathbb{S}_{1}^{n+1}\right)=\left\{v\in\mathbb{L}^{n+2}:\left\langle v,p\right\rangle=0\right\}.$$
	Let $a\in\mathbb{L}^{n+2}\setminus\{0\}$ be a past-pointing null vector, that is, $\left\langle a,a\right\rangle=0$ and $\left\langle a,e_{n+2}\right\rangle>0$, where $e_{n+2}=(0,\dots,0,1)$. Then, the open region of the de Sitter space $\mathbb{S}_{1}^{n+1}$, given by
	$$\mathcal{H}^{n+1}=\left\{p\in\mathbb{S}_{1}^{n+1}:\left\langle p,a\right\rangle >0\right\}$$
	is the so-called {\em steady state space}.
	
	The importance of considering $\mathcal{H}^{n+1}$ comes from the fact that, in Cosmology, $\mathcal H^4$ is the steady state model of the universe proposed by Bondi-Gold~\cite{bgold} and Hoyle~\cite{hoyle}, when looking for a model of the universe which looks the same not only at all points and in all directions (that is, spatially isotropic and homogeneous), but also at all times. For more details, we recommend for the readers to see Section 5.2 of \cite{hawking} or Section 14.8 of \cite{Sw}.
	
	From a mathematical point of view, the interest in the study of spacelike hypersurfaces immersed in a Lorentzian space is motivated by their nice Bernstein-type properties. In this direction, several authors have approached the problem of to characterizing spacelike hyperplanes of $\mathcal H^{n+1}$, which are totally umbilical spacelike hypersurfaces isometric to the Euclidean space $\mathbb R^n$ and give a complete foliation of $\mathcal H^{n+1}$. We refer to readers, for instance, the works~\cite{albujer,Aquino:15,CH10,CH12,sm03}.
	
	Proceeding with this picture, our purpose in this article is to apply a suitable extension of the generalized maximum principle of Omori~\cite{Omori:67} and Yau~\cite{Yau:75} due to Al\'{\i}as, Impera and Rigoli in~\cite{Alias:12} (see Lemma~\ref{lemma:Maximum Principle}) in order to establish rigidity results for complete spacelike hypersurfaces of $\mathcal{H}^{n+1}$. For this, we will assume certain appropriate constraints on the behavior of the higher order mean curvatures. In this setting, We also obtain some nonexistence results concerning these spacelike hypersurfaces.
	
	This paper is organized in the following manner: initially, in Section~\ref{sec:preliminaries}, we recall some basic facts concerning the higher order mean curvatures, the Newton transformations and the linearized operators naturally attached to a spacelike hypersurface immersed in $\mathcal{H}^{n+1}$. In the last part of this same section, we also quote some key lemmas that will be essential to prove our main results. Afterwards, in Section~\ref{sec:main results} we establish our rigidity and nonexistence results for complete spacelike hypersurfaces of $\mathcal{H}^{n+1}$ (see Theorems~\ref{thm1-H},~\ref{thm2-H} and~\ref{thm3-H}, and Corollaries~\ref{crl1},~\ref{nonexistence:1} and~\ref{nonexistence:2}). Finally, in Section~\ref{sec:further results} we obtain further rigidity results considering that the height and angle functions of the spacelike hypersurfaces are linearly related (see Theorems~\ref{thm4-H} and~\ref{thm5-H}, and Corollary~\ref{crl2-H}).
	
	\section{Preliminaries and key lemmas}\label{sec:preliminaries}
	
	Along this article, we will deal with connected spacelike hypersurfaces $\psi:\Sigma\rightarrow\mathcal{H}^{n+1}$ with a future-pointing normal unitary timelike vector field $N$. Let us denote by $A:\mathfrak X(\Sigma)\rightarrow\mathfrak X(\Sigma)$ the Weingarten endomorphism of $\Sigma^n$ with respect to $N$. We recall that at each point $p\in\Sigma^{n},$ the Weingarten operator $A$ restricts to a self-adjoint linear map $A_{p}:T_{p}\Sigma\rightarrow T_{p}\Sigma.$ For $0\leq r\leq n,$ let $S_{r}(p)$ denote the $r$-th elementary symmetric function on the eigenvalues of $A_{p}.$ Hence, one gets $n$ smooths functions $S_{r}:\Sigma\rightarrow\mathbb{R},$ such that
	$$\det(tI-A)=\sum_{k=0}^n(-1)^kS_kt^{n-k},$$
	where $S_0=1$ by convention. If $p\in\Sigma^n$ and $\{e_k\}$ is a basis of $T_p\Sigma$ formed by eigenvectors of $A_p$, with corresponding eigenvalues $\{\lambda_k\}$, one immediately sees that
	$$S_r=\sigma_r(\lambda_1,\ldots,\lambda_n),$$
	where $\sigma_r\in\mathbb R[X_1,\ldots,X_n]$ is the $r$-th elementary symmetric polynomial on the indeterminates $X_1,\ldots,X_n$.
	
	In the sequel, with this setting, we define the $r$-th mean curvature $H_r$ of $\Sigma^n$, $0\leq r\leq n$, by
	\begin{equation}\label{Hr}
		{n\choose r}H_r=(-1)^{r}S_r=\sigma_r(-\lambda_1,\ldots,-\lambda_n).
	\end{equation}
	We observe that $H_0=1$, while $H_1=-\frac{1}{n}S_1$ is the usual mean curvature $H$ of $\Sigma^n$.
	
	For $0\leq r\leq n$, one defines the $r$-th Newton transformation $P_r$ on $\Sigma^n$ by setting $P_0=I$ (the identity operator) and, for $1\leq r\leq n$, via the recurrence relation
	\begin{equation}\label{eq:Newton operators}
		P_r={n\choose r}H_rI+AP_{r-1}.
	\end{equation}
	With a trivial induction, from \eqref{eq:Newton operators} we verify that
	\begin{equation}\label{eq:Newton operators2}
		P_r={n\choose r}H_rI + {n\choose r-1}H_{r-1}A+{n\choose r-2}H_{r-2}A^2+\cdots+A^r,
	\end{equation}
	so that Cayley-Hamilton theorem gives $P_n=0$. Moreover, since $P_r$ is a polynomial in $A$ for every $r$, it is also self-adjoint and commutes with $A$. Therefore, all bases of $T_p\Sigma$ diagonalizing $A$ at $p\in\Sigma^n$ also diagonalize all of the $P_r$ at $p$. So, let $\{e_1,\ldots,e_n\}$ be an orthonormal frame on $T_p\Sigma$ which diagonalizes $A_p$, $A_p(e_i)=\lambda_i(p)e_i$, then from \eqref{eq:Newton operators2} we have that
	\begin{equation}\label{spectrum of Pr}
		(P_r)_pe_i=(-1)^r\sum_{i_1<\ldots<i_r,i_j\neq i}\lambda_{i_1}(p)\ldots\lambda_{i_r}(p)e_i.
	\end{equation}
	Moreover, it is not difficult to check that $P_re_i=(-1)^rS_r(A_i)e_i$ and, consequently, we obtain the following lemma (see Lemma $2.1$ of~\cite{Barbosa:97}).
	
	\begin{lemma}\label{lemma:traces formulas} With the above notations, the following formulas hold:
		\begin{enumerate}
			\item[(a)] $S_r(A_i)=S_r-\lambda_iS_{r-1}(A_i)$;
			\item[(b)] ${\rm tr}(P_r)=(-1)^r\displaystyle\sum_{i=1}^nS_r(A_i)=(-1)^r(n-r)S_r=c_rH_r$;
			\item[(c)] ${\rm tr}(AP_r)=(-1)^r\displaystyle\sum_{i=1}^n\lambda_iS_r(A_i)=(-1)^r(r+1)S_{r+1}=-c_rH_{r+1}$;
			\item[(d)] ${\rm tr}(A^2P_r)=(-1)^r\displaystyle\sum_{i=1}^{n}\lambda_i^2 S_r(A_i)=\displaystyle{{n\choose r+1}}(nHH_{r+1}-(n-r-1)H_{r+2})$,
		\end{enumerate}
		where $c_r=(n-r)\displaystyle{n\choose r}$.
	\end{lemma}
	
	Associated to each Newton transformation $P_{r}$ one has the second-order linear differential operator $L_{r}$, defined by
	\begin{equation}\label{definition:Lr}
		L_{r}f={\rm tr}(P_{r}\nabla^2f),
	\end{equation}
	where $\nabla^2f:\mathfrak{X}(\Sigma)\rightarrow\mathfrak{X}(\Sigma)$ denotes the self-adjoint linear operator metrically equivalent to the Hessian of $f$, which is given by
	$$\langle\nabla^2f(X),Y\rangle=\langle\nabla_X\nabla f,Y\rangle,$$
	for all $X,Y\in\mathfrak{X}(\Sigma)$.
	
	In particular, for $r=0,$ we get the well known Laplacian operator $L_{0}=\Delta$, which is always elliptic. The next lemma gives a geometric condition which guarantees the ellipticity of $L_{1}$ (cf. Lemma 3.2 of~\cite{ac07}).
	
	\begin{lemma}\label{lemma:positiveness of P1}
		Let $\psi:\Sigma^n\rightarrow\mathcal{H}^{n+1}$ be a spacelike hypersurface in $\mathcal{H}^{n+1}$. If $H_2>0$ on $\Sigma^n$, then $L_1$ is elliptic or, equivalently, $P_1$ is positive definite (for an appropriate choice of orientation $N$).
	\end{lemma}
	
	When $r\geq 2$, the following lemma establishes sufficient conditions to guarantee the ellipticity of $L_r$ (cf. Lemma 3.3 of~\cite{ac07}).
	
	\begin{lemma}\label{lemma:ellipticity of Lr}
		Let $\psi:\Sigma^n\rightarrow\mathcal{H}^{n+1}$ be a spacelike hypersurface in $\mathcal{H}^{n+1}$. If there exists an elliptic point of $\Sigma^n$, with respect to an appropriate choice of orientation $N$, and $H_{r+1}>0$ on $\Sigma^n$, for $2\leq r\leq n-1$, then for all $1\leq k\leq r$ the operator $L_k$ is elliptic or, equivalently, $P_k$ is positive definite (for an appropriate choice of orientation $N$, if $k$ is odd).
	\end{lemma}
	
	Here, by an elliptic point in a spacelike hypersurface $\psi:\Sigma^n\rightarrow\mathcal{H}^{n+1}$ we mean a point $p_0\in\Sigma^n$ where all principal curvatures $\lambda(p_0)$ are negative.
	
	The next lemma was done by Al\'{\i}as, Brasil Jr. and Colares~\cite{ABC:03} in a more general setting, when they studied spacelike hypersurfaces in conformally stationary spacetime (see Lemma $5.4$ of~\cite{ABC:03}). Taking into account our purposes, we rewrote it as follows.
	
	\begin{lemma}\label{lemma:existence of an elliptic point}
		Let $V$ be a complete closed conformal timelike vector field globally defined on the steady state space $\mathcal{H}^{n+1}$, and let $\Sigma^n$ be a complete spacelike hypersurface in $\mathcal{H}^{n+1}$. Suppose that the divergence of $V$ on $\mathcal{H}^{n+1}$, ${\rm Div}V$, does not vanish at a point of $\Sigma^n$ where the restriction $|V|_{\Sigma}=\sqrt{-\langle V,V\rangle}|_{\Sigma}$ of $|V|$ to $\Sigma^n$ attains a local minimum. Then, there exists an elliptic point $p_0\in\Sigma^n$.
	\end{lemma}
	
	We close this section quoting the maximum principle which will be used to obtain our rigidity and nonexistence results in the next section. Let $\Sigma^n$ be a complete Riemannian manifold and let $\mathcal P:\mathfrak{X}(\Sigma)\rightarrow\mathfrak{X}(\Sigma)$ denote a self-adjoint operator. Extending the idea of the linearized operator $L_{r}$ defined in \eqref{definition:Lr}, we consider a second order linear differential operator $\mathcal L:C^{\infty}(\Sigma)\rightarrow C^{\infty}(\Sigma)$ naturally associated to $\mathcal P$, given by
	
	\begin{equation}\label{operator L}
		\mathcal Lf=\,\text{tr}\,(\mathcal P\nabla^2f).
	\end{equation}
	
	From Lemma $4.2$ of \cite{Alias:12}, we have the following generalized maximum principle related to such an operator $\mathcal L$, which is an extension of the generalized maximum principle of Omori~\cite{Omori:67} and Yau~\cite{Yau:75} (see also~\cite{Alias-Mastrolia-Rigoli:16} for a modern and accessible reference to the generalized maximum principle of Omori-Yau).
	
	\begin{lemma}\label{lemma:Maximum Principle}
		Let $\Sigma^n$ be a complete Riemannian manifold with sectional curvature bounded from below, and $f\in C^{\infty}(\Sigma)$ be a function which is bounded from above on $\Sigma^n$. If $\mathcal P$ is positive semi-definite and ${\rm tr}(\mathcal P)$ is bounded from above on $\Sigma^n$, then there exists a sequence $(p_k)_{k\geq 1}$ in $\Sigma^n$ such that
		\begin{equation*}
			\lim_kf(p_k)=\sup_{\Sigma}f,\,\,\,\,\,\lim_k|\nabla f(p_k)|=0\,\,\,\,\,{\rm and}\,\,\,\,\limsup_k\mathcal Lf(p_k)\leq0,
		\end{equation*}
		where the operator $\mathcal L$ is given by \eqref{operator L}.
	\end{lemma}
	
	\section{Rigidity and nonexistence results in $\mathcal{H}^{n+1}$}\label{sec:main results}
	
	As introduced before, the steady state space $\mathcal{H}^{n+1}$ is the hyperquadric
	$$\mathcal{H}^{n+1}=\{x\in\mathbb{S}^{n+1}_{1};\langle x,a\rangle>0\},$$
	with $a\in\mathbb{L}^{n+2}$ be a past-pointing null vector, that is, $\langle a,a\rangle=0$ and $\langle a,e_{n+2}\rangle>0,$ where $e_{n+2}=(0,\ldots,0,1)$.
	
	According to Example 2 in Section 4 of \cite{sm99}, the timelike vector field
	$$V=\langle x,a\rangle x+a$$
	is such that
	\begin{equation}\label{campo-conforme}
		\overline{\nabla}_{W}V=\langle x,a\rangle W,
	\end{equation}
	for all $W\in\mathfrak{X}(\mathcal{H}^{n+1})$, where $\overline{\nabla}$ stands for the Levi-Civita connection of $\mathcal{H}^{n+1}$. Thus, $V$ is a complete closed conformal timelike vector field globally defined on $\mathcal{H}^{n+1}$. Proposition 1 of \cite{sm99} guarantees that the $n$-dimensional distribution $\mathcal{D}(V)$ defined on $\mathcal{H}^{n+1}$ by
	$$x\in\mathcal{H}^{n+1}\longmapsto\mathcal{D}(x)=\{v\in T_{x}\mathcal{H}^{n+1}:\langle V(x),v\rangle=0\}$$
	determines a codimension one spacelike foliation $\mathcal{F}(V)$ which is oriented by $V$. Furthermore, the leaves of $\mathcal{D}(V)$ are given by
	$$\mathcal{E}_{\tau}=\{x\in\mathcal{H}^{n+1}:\langle x,a\rangle=\tau\},\,\,\,{\rm with}\,\,\,\tau>0,$$
	which are totally umbilical hypersurfaces of $\mathcal{H}^{n+1}$ isometric to the $n$-dimensional Euclidean space $\mathbb{R}^{n}$, having constant $r$-th mean curvature $H_r=1$ (with respect to the unit normal fields $N_{\tau}=x-\dfrac{1}{\tau}a,\;\;\;x\in\mathcal{L}_{\tau}$, when $r$ is odd).
	
	\begin{center}
		\centering
		\includegraphics[height=5cm]{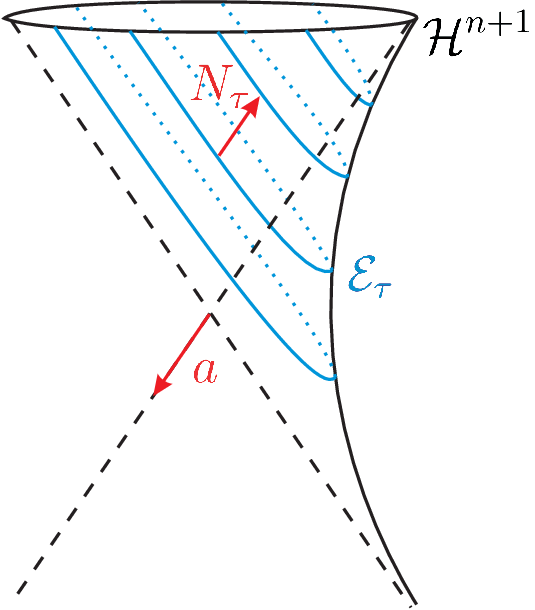} \\
		Figure 1: Foliating $\mathcal{H}^{n+1}$ via spacelike hyperplanes $\mathcal{E}_{\tau}$.
	\end{center}
	
	In this setting, we will consider two particular functions naturally attached to a spacelike hypersurface $\Sigma^{n}$ immersed on $\mathcal{H}^{n+1}$, namely, the height and angle functions with respect to a previously fixed nonzero null vector $a\in\mathbb{L}^{n+2}$, which are defined, respectively, by $l_{a}=\langle \psi,a \rangle$ and $f_{a}=\langle N,a\rangle$. It is not difficult to check that $\nabla l_{a}=a^{\top}$ and $\nabla f_{a}=-A(a^{\top}),$ where $a^{\top}$ stands for the orthogonal projection of $a$ onto the tangent bundle $T\Sigma$. Moreover, using Gauss and Weingarten formulas, we obtain
	\begin{equation}\label{Hessian of la}
		\nabla_{X}\nabla l_{a}=-f_{a}AX-l_{a}X,
	\end{equation}
	for all $X\in\mathfrak{X}(\Sigma)$. From Lemma~\ref{lemma:traces formulas} jointly with \eqref{Hessian of la}, we can deduce the formula for the operator $L_{r}$ acting on the height function, that is,
	\begin{equation}\label{Lrla}
		L_{r}l_{a}=c_{r}(-l_{a}H_{r}+f_{a}H_{r+1}),
	\end{equation}
	where $c_r=(n-r){n\choose r}$.
	
	In what follows, we say that a spacelike hypersurface $\Sigma^n$ immersed in $\mathcal{H}^{n+1}$ is contained in the closure of the interior domain enclosed by a spacelike hyperplane $\mathcal E_\tau$, with $\tau>0$, when its height function satisfies $l_a\leq\tau$ (see Figure 2).
	
	\begin{center}
		\centering
		\includegraphics[height=5cm]{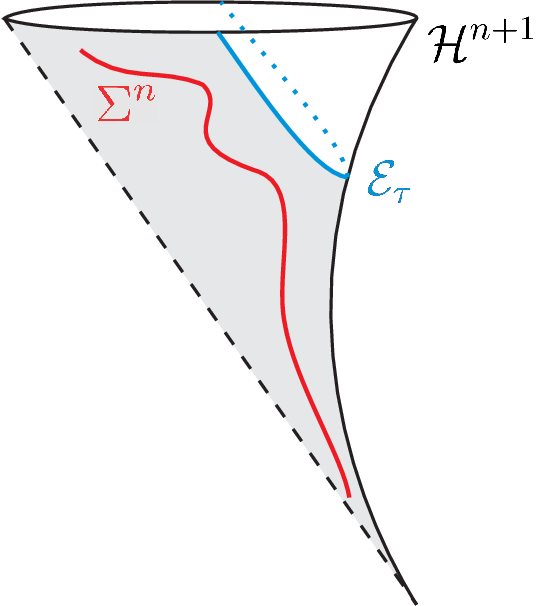} \\
		Figure 2: $\Sigma^n$ is contained in the closure of the interior domain enclosed by $\mathcal E_\tau$.
	\end{center}
	
	Now, we are in position to state and prove our first rigidity result for complete spacelike hypersurfaces $\psi:\Sigma^{n}\rightarrow\mathcal{H}^{n+1}$. Fixed a past-pointing nonzero null vector $a\in\mathbb L^{n+2}$ as before and taking a future-pointing orientation $N$ for such a $\psi$, along this paper we will assume that its angle function $f_{a}$ is positive.
	
	\begin{theorem}\label{thm1-H}
		Let $\psi:\Sigma^{n}\rightarrow\mathcal{H}^{n+1}$ be a complete spacelike hypersurface of $\mathcal{H}^{n+1}$ contained in the closure of the interior domain enclosed by a spacelike hyperplane $\mathcal E_{\tau}$ orthogonal to a nonzero null vector $a\in\mathbb{L}^{n+2}$. Suppose that the mean curvature $H$ of $\Sigma^{n}$ is positive, bounded and satisfies
		\begin{equation}\label{hipot1-teo1}
			H\leq H_2.
		\end{equation}
		If
		\begin{equation}\label{hipot2-teo1}
			|a^\top| \leq C \inf_{\Sigma} (H_2 - H),
		\end{equation}
		for some positive constant $C$, then $\Sigma^{n}$ is a spacelike hyperplane $\mathcal E_{\tilde{\tau}}$ with $\tilde{\tau}\leq\tau$.
	\end{theorem}
	
	\begin{proof}
		Regarding that $a^{\top}$ stands for the orthogonal projection of $a$ onto the tangent bundle $T\Sigma$, we have that $a^{\top}=a+f_{a}N-l_{a}\psi$. Consequently,
		\begin{equation}\label{normaa_0}
			|a^{\top}|^{2}=f_{a}^{2}-l_{a}^{2}.
		\end{equation}
		In particular, we obtain that $f_{a}\geq l_{a}>0$. From (\ref{Lrla}) we can see that
		\begin{equation}\label{eqaux:1_0}
			L_1(\,l_a\,)=-c_{r}Hl_{a}+c_{r}H_{2}f_{a}\geq c_{1}( H_2 - H)l_a,
		\end{equation}
		where $c_1=n(n-1)$.
		
		From Cauchy-Schwarz inequality we have that $H_2 \leq H^2$ and, since we are assuming that $H$ is bounded, we get that $H_2$ is also bounded. Thus, taking into the algebraic relation
		\begin{equation}\label{norm of the shape operator}
			\sum_{i=1}^n\lambda_i^2=|A|^2=n^2H^2-n(n-1)H_2,
		\end{equation}
		we have that all the principal curvatures $\lambda_i$ of $\Sigma^n$ are bounded. Consequently, from Gauss equation
		\begin{equation}\label{Gauss equation}
			K_{\Sigma}(e_i,e_j)=1 - \lambda_{i}\lambda_{j},
		\end{equation}
		we conclude that the sectional curvature $K_{\Sigma}$ of $\Sigma^n$ is bounded from below.
		
		We note that our assumption that $\Sigma^n$ is contained in the closure of the interior domain enclosed by $\mathcal E_{\tau}$ guarantees that $l_a$ is bounded. On the other hand, using hypothesis \eqref{hipot1-teo1}, it follows from Lemma $3.10$ of~\cite{Elbert:02} that $L_1$ is elliptic and, in particular, $P_1$ is positive semi-definite and ${\rm tr}(P_1) = c_1 H$ is bounded.
		
		Hence, since \eqref{eqaux:1_0} gives
		\begin{equation*}
			L_1(\,l_a\,)\geq n(n-1)\left(H_2-H\right)l_a\geq0.
		\end{equation*}
		we can apply Lemma~\ref{lemma:Maximum Principle} to obtain a sequence $(p_k)_{k\geq 1}$ in $\Sigma^n$ such that
		\begin{equation}\label{eq.3BB}
			\lim_kl_a(p_k)=\sup_{\Sigma}l_a,\,\,\,\,\,\lim_k|\nabla l_a(p_k)|=0\,\,\,\,\,{\rm and}\,\,\,\,\limsup_k L_1(\,l_a\,)(p_k)\leq0.
		\end{equation}
		Consequently, from \eqref{eqaux:1_0} and \eqref{eq.3BB} we have that
		\begin{equation}\label{eq.3CC}
			0\geq\limsup_k L_1(\,l_a\,)(p_k)\geq n(n-1) (\sup_{\Sigma}l_a)\limsup_k\left(H_2 -H \right)(p_k)\geq0.
		\end{equation}
		Thus, since $\sup_{\Sigma}l_a>0$, from \eqref{eq.3CC} we get that
		$$\limsup_k\left(H_2 -H \right)(p_k)=0$$
		and, in particular,
		\begin{equation}\label{eqaux:9}
			\inf_{\Sigma}\left(H_2 -H\right)=0.
		\end{equation}
		Therefore, hypothesis \eqref{hipot2-teo1} jointly with \eqref{eqaux:9} guarantee that $a^\top=\nabla l_a$ vanishes identically on $\Sigma^n$, that is, $l_a$ is constant on $\Sigma^n$, which implies that $\Sigma^n$ is a spacelike hyperplane $\mathcal E_{\tilde{\tau}}$ with $\tilde{\tau} \leq \tau$.
	\end{proof}
	
	Taking into account that $H_2 = 1 - R$, where $R$ is the normalized scalar curvature of the hypersurface, from Theorem \eqref{thm1-H} we obtain the following consequence:
	
	\begin{corollary}\label{crl1}
		Let $\psi:\Sigma^{n}\rightarrow\mathcal{H}^{n+1}$ be a complete spacelike hypersurface of $\mathcal{H}^{n+1}$ contained in the closure of the interior domain enclosed by a hyperplane $\mathcal E_{\tau}$ orthogonal to a nonzero null vector $a\in\mathbb{L}^{n+2}$. Suppose that the mean curvature $H$ of $\Sigma^{n}$ is positive, bounded and satisfies
		\begin{equation*}
			H+R\leq1,
		\end{equation*}
		where $R$ is the normalized scalar curvature of $\Sigma^n$. If
		\begin{equation*}\label{hipot2-cor1}
			|a^\top|\leq C\{1-\sup_\Sigma(H+R)\},
		\end{equation*}
		for some positive constant $C$, then $\Sigma^{n}$ is a spacelike hyperplane $\mathcal E_{\tilde{\tau}}$ with $\tilde{\tau} \leq \tau$.
	\end{corollary}
	
	In the next result, we consider the rigidity via suitable constraints on the higher order mean curvatures.
	
	\begin{theorem}\label{thm2-H}
		Let $\psi:\Sigma^{n}\rightarrow\mathcal{H}^{n+1}$ be a complete spacelike hypersurface of $\mathcal{H}^{n+1}$ contained in the closure of the interior domain enclosed by a spacelike hyperplane $\mathcal E_{\tau}$ orthogonal to a nonzero null vector $a\in\mathbb{L}^{n+2}$, with sectional curvature $K_{\Sigma}\leq1$ and bounded from below. Suppose that, for some $1\leq r\leq n-1$, $H_{r+1}$ is bounded and satisfies
		\begin{equation}\label{hipot1-teo2}
			\beta\leq H_{r}\leq H_{r+1},
		\end{equation}
		where $\beta$ is a positive constant. If
		\begin{equation}\label{hipot2-teo2}
			|a^\top| \leq C \inf_{\Sigma} (H_{r+1} - H_r),
		\end{equation}
		for some positive constant $C$, then $\Sigma^{n}$ is a spacelike hyperplane $\mathcal E_{\tilde{\tau}}$ with $\tilde{\tau} \leq \tau$.
	\end{theorem}
	
	\begin{proof}
		Since \eqref{normaa_0} guarantees that $f_a\geq l_a$, from (\ref{Lrla}) we get that
		\begin{equation}\label{eqaux:1}
			L_r(\,l_a\,)=-c_{r}H_{r}l_{a}+c_{r}H_{r+1}f_{a}\geq c_{r}( H_{r+1} - H_r)l_a,
		\end{equation}
		where $c_r=(n-r){n\choose r}$.
		
		We define on $\Sigma^n$ the following self-adjoint operator $\mathcal{P}_r : \mathfrak{X}(\Sigma^n) \longrightarrow \mathfrak{X}(\Sigma^n)$ by
		\begin{equation}\label{Pfresco}
			\mathcal{P}_r := H_r P_r.
		\end{equation}
		Taking a (local) orthonormal frame $\{e_{1},\ldots,e_{n}\}$ such that $Ae_{i}=\lambda_{i}e_{i}$, from \eqref{Hr} and \eqref{spectrum of Pr} we have that
		\begin{eqnarray}\label{eqaux:2}
			\langle\mathcal P_re_{i},e_{i}\rangle={n\choose r}^{-1}\sum_{i_1<\ldots<i_r,i_j\neq i,j_1<\ldots<j_r}(\lambda_{i_{1}}\lambda_{j_{1}})\ldots(\lambda_{i_{r}}\lambda_{j_{r}}).
		\end{eqnarray}
		But, since we are assuming that $K_{\Sigma}\leq1$, from Gauss equation we obtain
		\begin{equation}\label{curvaturas}
			\lambda_i \lambda_j = 1 - K_{\Sigma}(e_i, e_j) \geq 0 ,
		\end{equation}
		for all $i,j \in \{ 1, \ldots, n \}$, with $i \neq j$. Thus, from \eqref{eqaux:2} and \eqref{curvaturas} we get
		\begin{equation*}\label{positivo-definido}
			\langle \mathcal{P}_r e_i, e_i \rangle \geq 0.
		\end{equation*}
		Consequently, $\mathcal P_r$ is positive semi-definite. In addition, since we are also assuming that $H_r$ is bounded on $\Sigma^{n}$, we have that the same happens for ${\rm tr}(\mathcal P_r)=c_{r}H_{r}^{2}$.
		
		Extending the idea of the proof of Theorem~\ref{thm2-H}, we consider the operator $\mathcal L_r:C^{\infty}(\Sigma)\rightarrow C^{\infty}(\Sigma)$ given by
		\begin{equation}\label{eq:1-10}
			\mathcal L_rf={\rm tr}(\mathcal P_r\nabla^2f).
		\end{equation}
		Since $\mathcal P_r$ is positive semi-definite, from \eqref{hipot1-teo2}, \eqref{eqaux:1} and \eqref{eq:1-10} we get
		\begin{equation}\label{eq.3}
			\mathcal L_r(\,l_a\,)\geq c_r\left(H_{r+1} - H_r\right)l_a H_r \geq0.
		\end{equation}
		So, taking into account that our assumption that $\Sigma^n$ is contained in the closure of the interior domain enclosed by $\mathcal E_{\tau}$ implies that $l_a$ is bounded, we can apply Lemma~\ref{lemma:Maximum Principle} to obtain a sequence $(p_k)_{k\geq 1}$ in $\Sigma^n$ such that
		\begin{equation}\label{eq.3B}
			\lim_kl_a(p_k)=\sup_{\Sigma}l_a,\,\,\,\,\,\lim_k|\nabla l_a(p_k)|=0\,\,\,\,\,{\rm and}\,\,\,\,\limsup_k\mathcal L_r(\,l_a\,)(p_k)\leq0.
		\end{equation}
		Consequently, from \eqref{eq.3} and \eqref{eq.3B} we have that
		\begin{equation}\label{eq.3C}
			0\geq\limsup_k\mathcal L_r(\,l_a\,)(p_k)\geq c_r \beta (\sup_{\Sigma}l_a) \limsup_k\left(H_{r+1}-H_r\right)(p_k)\geq0.
		\end{equation}
		Hence, since $\displaystyle\sup_{\Sigma}l_a>0$, from \eqref{eq.3C} we get that
		$$\limsup_k\left(H_{r+1}-H_r \right)(p_k)=0$$
		and, in particular,
		\begin{equation}\label{eqaux:10}
			\inf_{\Sigma}\left(H_{r+1}-H_r\right)=0.
		\end{equation}
		Therefore, hypothesis \eqref{hipot2-teo2} guarantees that $\nabla l_a=a^\top$ vanishes identically on $\Sigma^n$, that is, $l_a$ is constant on $\Sigma^n$, which implies that $\Sigma^n$ is a spacelike hyperplane $\mathcal E_{\tilde{\tau}}$ with $\tilde{\tau} \leq \tau$.
	\end{proof}
	
	From the proof of Theorem \eqref{thm2-H}, we obtain the following nonexistence result:
	
	\begin{corollary}\label{nonexistence:1}
		There do not exist complete spacelike hypersurfaces $\psi:\Sigma^{n}\rightarrow\mathcal{H}^{n+1}$ contained in the closure of the interior domain enclosed by a spacelike hyperplane $\mathcal E_{\tau}$ orthogonal to a nonzero null vector $a\in\mathbb{L}^{n+2}$, with sectional curvature $K_{\Sigma}\leq1$ and bounded from below such that, for some $1\leq r\leq n-1$, $H_r$ and $H_{r+1}$ are positive constant satisfying $H_{r}<H_{r+1}$.
	\end{corollary}
	
	\begin{proof}
		Suppose by contradiction that there is such a spacelike hypersurface $\Sigma^n$. But, from the proof of Theorem \eqref{thm2-H} we obtain that
		\begin{equation}
			\inf_{\Sigma}\left(H_{r+1}-H_r\right)=0,
		\end{equation}
		which implies that $H_r = H_{r+1}$, contradicting our hypothesis that $H_{r+1} > H_r$.
	\end{proof}	
	
	Before presenting our next result, we will need to establish the following definition. We say that an immersed hypersurface $\Sigma^{n}$ in $\mathcal H^{n+1}$ is {\em locally tangent from above} to a spacelike hyperplane $\mathcal E_{\widetilde{\tau}}$ of $\mathcal H^{n+1}$, when there exist a point $p\in\Sigma^n$ and a neighborhood $\mathcal U\subset\Sigma^n$ of $p$ such that $l_a(p)=\widetilde{\tau}$ and $l_a(q)\geq\widetilde{\tau}$ for all $q\in\mathcal U$ (see Figure 3).
	
	\begin{center}
		\centering
		\includegraphics[height=5cm]{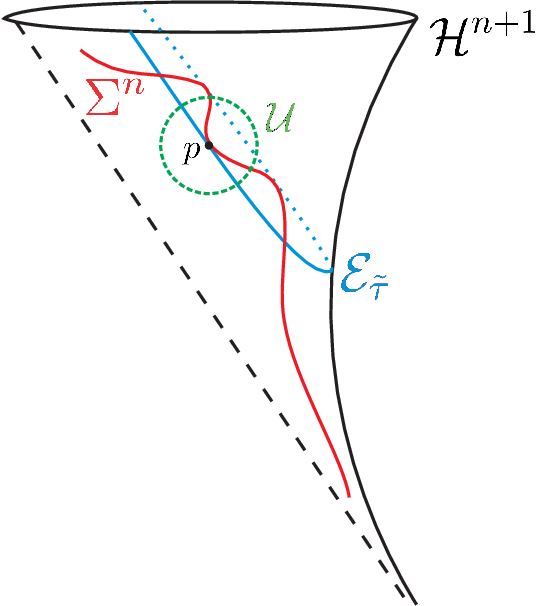} \\
		Figure 3: $\Sigma^n$ is locally tangent from above to a spacelike hyperplane $\mathcal E_{\tilde{\tau}}$.
	\end{center}
	
	In this setting, we obtain the following rigidity result:
	
	\begin{theorem}\label{thm3-H}
		Let $\psi:\Sigma^{n}\rightarrow\mathcal{H}^{n+1}$ be a complete spacelike hypersurface of $\mathcal{H}^{n+1}$ contained in the closure of the interior domain enclosed by a spacelike hyperplane $\mathcal E_\tau$ orthogonal to a nonzero null vector $a \in \mathbb{L}^{n+2}$, and locally tangent from above to a spacelike hyperplane $\mathcal E_{\tilde{\tau}}$, with $\tilde{\tau} \leq \tau$. Suppose that $H$ is bounded and, for some $1 \leq r \leq n - 1$, $H_{r+1}$ is positive and such that
		\begin{equation}\label{H1-teor3}
			H_r \leq H_{r+1}.
		\end{equation}
		If
		\begin{equation}\label{H2-teo3}
			|a^\top| \leq C \inf_{\Sigma} (H_{r+1} - H_r),
		\end{equation}
		for some positive constant $C$, then $\Sigma^{n}$ must be the spacelike hyperplane $\mathcal E_{\tilde{\tau}}$.
	\end{theorem}
	
	\begin{proof}
		Let as consider the vector field on $\mathcal{H}^{n+1}$ defined in the beginning of this section, namely $V(p) = - \langle p, a \rangle p + a$. It is not difficult to verify that $|V|_\Sigma = l_a$ and ${\rm Div } V(p) = (n+1) \langle p, a \rangle$. Thus, since we are supposing that $\Sigma^n$ is locally tangent from above to a spacelike hyperplane $\mathcal E_\tau$, we have that $|V|_\Sigma$ attains a minimum local on $\Sigma^n$. Consequently, we can apply Lemma \ref{lemma:existence of an elliptic point} to guarantee the existence of an elliptic point on $\Sigma^n$. Hence, since we are also supposing that $H_{r+1} > 0$, it follows from Lemma \ref{lemma:ellipticity of Lr} that $P_j$ is positive definite and, since ${\rm tr}(P_j) = c_j H_j$, $H_j$ is positive for all $1 \leq j \leq r$. Moreover, taking into account once more (\ref{norm of the shape operator}) and that $H_2 > 0$, we get that
		$$\sum_{i} \lambda_i^2 \leq n^2H^2.$$
		Consequently, the boundedness of $H$ implies in the boundedness of all principal curvatures of $\Sigma^n$. In particular, we have that $H_r$ is bounded and, from Gauss equation (\ref{Gauss equation}), $K_\Sigma$ is bounded from bellow. Thus, from (\ref{Lrla}) we have that
		\begin{equation}\label{Lr-teo4}
			L_r(\,l_a\,)\geq c_r\left(H_{r+1}-H_r\right)l_a\geq0.
		\end{equation}
		
		Hence, we can apply Lemma~\ref{lemma:Maximum Principle} to obtain a sequence $(p_k)_{k\geq 1}$ in $\Sigma^n$ such that
		\begin{equation}\label{eq.3BB_1}
			\lim_kl_a(p_k)=\sup_{\Sigma}l_a,\,\,\,\,\,\lim_k|\nabla l_a(p_k)|=0\,\,\,\,\,{\rm and}\,\,\,\,\limsup_k L_r(\,l_a\,)(p_k)\leq0.
		\end{equation}
		Consequently, from \eqref{eqaux:1_0} and \eqref{eq.3BB_1} we have that
		\begin{equation}\label{eq.3CC_1}
			0\geq\limsup_k L_r(\,l_a\,)(p_k)\geq c_r (\sup_{\Sigma}l_a)\limsup_k\left(H_{r+1} -H_r \right)(p_k)\geq0.
		\end{equation}
		Hence, since $\displaystyle\sup_{\Sigma}l_a>0$, from \eqref{eq.3CC_1} we get that
		$$\limsup_k\left(H_{r+1} -H_r \right)(p_k)=0$$
		and, in particular,
		\begin{equation}\label{eqaux:9_1}
			\inf_{\Sigma}\left(H_{r+1} -H_r \right)=0.
		\end{equation}
		Therefore, hypothesis \eqref{H2-teo3} jointly with \eqref{eqaux:9_1} guarantee that $\nabla l_a=a^\top$ vanishes identically on $\Sigma^n$, that is, $l_a$ is constant on $\Sigma^n$, which implies that $\Sigma^n$ is the spacelike hyperplane $\mathcal E_{\tilde{\tau}}$.
	\end{proof}	
	
	From the proof of Theorem~\ref{thm3-H} we obtain the following nonexistence result:
	
	\begin{corollary}\label{nonexistence:2}
		There do not exist complete spacelike hypersurfaces $\psi:\Sigma^{n}\rightarrow\mathcal{H}^{n+1}$ contained in the closure of the interior domain enclosed by a spacelike hyperplane $\mathcal E_\tau$ orthogonal to a nonzero null vector $a \in \mathbb{L}^{n+2}$, locally tangent from above to a spacelike hyperplane $\mathcal E_{\tilde{\tau}}$, with $\tilde{\tau} \leq \tau$, having bounded mean curvature and such that, for some $1 \leq r \leq n - 1$, $H_r$ and $H_{r+1}$ are positive constant satisfying $H_r < H_{r+1}$.
	\end{corollary}
	
	\section{Further rigidity results in $\mathcal H^{n+1}$}\label{sec:further results}
	
	Motivated by the fact that the spacelike hyperplanes $\mathcal E_\tau$ of $\mathcal H^{n+1}$ satisfy the condition $l_{a}=f_{a}$ (considering the unit normal fields $N_{\tau}=x-\frac{1}{\tau}a$), in the last section of this manuscript we present further rigidity results supposing that the height and angle functions of the spacelike hypersurfaces are linearly related. We start proving the following:
	
	\begin{theorem}\label{thm4-H}
		Let $\psi:\Sigma^{n}\rightarrow\mathcal{H}^{n+1}$ be a complete spacelike hypersurface of $\mathcal{H}^{n+1}$ contained in the closure of the interior domain enclosed by a spacelike hyperplane $\mathcal E_\tau$ orthogonal to a nonzero null vector $a \in \mathbb{L}^{n+2}$. Suppose that $l_{a}=\lambda f_{a}$ for some positive constant $\lambda\in\mathbb{R}$, the mean curvature $H$ of $\Sigma^{n}$ is bounded and that
		\begin{equation}\label{Hip-teo5}
			H_2\geq 1.
		\end{equation}
		Then $\Sigma^n$ is a spacelike hyperplane $\mathcal E_{\tilde{\tau}}$ with $\tilde{\tau} \leq \tau$.
	\end{theorem}
	
	\begin{proof}
		Since $l_{a}=\lambda f_{a}$, we observe that
		\begin{equation}\label{eqaux:6A_5}
			|\nabla l_{a}|^{2}=f_{a}^{2}-l_{a}^{2}=(\lambda^{-2}-1)l_{a}^{2}.
		\end{equation}
		In particular, from \eqref{eqaux:6A_5} we have that $\lambda^{-2}-1\geq0$. Now, we define on $\Sigma^n$ the following operator $\mathcal{L} : C^{\infty}(\Sigma) \longrightarrow C^{\infty}(\Sigma)$ by
		\begin{equation}\label{Operador L_5}
			\mathcal{L} f = \displaystyle\frac{\lambda^{-1}}{n(n-1)}L_1 f +\frac{1}{n}\Delta f,
		\end{equation}
		From equation \eqref{Lrla} and the hypothesis \eqref{Hip-teo5}, we obtain that
		\begin{equation}\label{lambda-equation_5}
			\begin{aligned}
				\mathcal{L} (l_a)
				& = \dfrac{\lambda^{-1}}{n(n-1)} \{ n(n-1)( -l_a H_1 + f_a H_2 )  \} + \dfrac{1}{n} \{  n(  -l_a + f_a H_1 )  \} \\
				& = \lambda^{-2} H_2 l_a - l_a  \geq (\lambda^{-2} - 1) l_a \geq 0.
			\end{aligned}
		\end{equation}	
		Hence, since we are also supponsing that $H_2 > 0$, it follows from Lemma \ref{lemma:ellipticity of Lr} that $P_1$ is positive definite, consequently, $H$ is positive. Moreover, we know that
		\begin{equation}\label{traco de P_j_5}
			{\rm tr}(P_1) = c_1 H.
		\end{equation}
		Thus, taking into account once more that $H_2 \geq 1$ and $H_2 < H^2$, from \eqref{norm of the shape operator} we get that
		$$\sum_{i} \lambda_i^2 \leq n^2H^2.$$
		Consequently, the boundedness of $H$ implies in the boundedness of all principal curvatures of $\Sigma^n$. From Gauss equation (\ref{Gauss equation}), $K_\Sigma$ is bounded from bellow. Since our assumption that $\Sigma^n$ is contained in the closure of the interior domain enclosed by $\mathcal E_{\tau}$ implies that $l_a$ is bounded, we can apply Lemma~\ref{lemma:Maximum Principle} to get a sequence $(p_k)_{k\geq 1}$ in $\Sigma^n$ such that
		\begin{equation}\label{teo4-omori_5}
			\lim_kl_a(p_k)=\sup_{\Sigma}l_a,\,\,\,\,\,\lim_k|\nabla l_a(p_k)|=0\,\,\,\,\,{\rm and}\,\,\,\,\limsup_k\mathcal L(\,l_a\,)(p_k)\leq0.
		\end{equation}
		Hence, from \eqref{teo4-omori} and \eqref{lambda-equation}, we obtain that
		\begin{equation}\label{teo4-final_5}
			0\geq\limsup_k \mathcal{L}(\,l_a\,)(p_k)\geq (\lambda^2 - 1) (\sup_{\Sigma}l_a)\geq0.
		\end{equation}
		So, as $\sup_\Sigma l_a > 0$, then
		$$\lambda^{-2} - 1 = 0.$$
		Therefore, returning to \eqref{lambda-equation}, we obtain that $\lambda = 1$. From \eqref{eqaux:6A}, we have that $ \nabla l_a = 0$ and the height function $l_a$ is constant on $\Sigma^n$, consequently, $\Sigma^n$ is a spacelike hyperplane $\mathcal E_{\tilde{\tau}}$ for some $\tilde{\tau} \leq \tau$.
	\end{proof}	
	
	From Theorem~\ref{thm4-H} and using once more the relation between $H_2$ and the normalized scalar curvature, we obtain the following consequence: 
	
	\begin{corollary}\label{crl2-H}
		Let $\psi:\Sigma^{n}\rightarrow\mathcal{H}^{n+1}$ be a complete spacelike hypersurface of $\mathcal{H}^{n+1}$ contained in the closure of the interior domain enclosed by a spacelike hyperplane $\mathcal E_\tau$ orthogonal to a nonzero null vector $a \in \mathbb{L}^{n+2}$. Suppose that $l_{a}=\lambda f_{a}$, for some positive constant $\lambda\in\mathbb{R}$, the mean curvature is bounded and the normalized scalar curvature is nonpositive. Then, $\Sigma^n$ is a spacelike hyperplane $\mathcal E_{\tilde{\tau}}$ with $\tilde{\tau} \leq \tau$.
	\end{corollary}	
	
	We close this paper extending Theorem~\ref{thm4-H} for the context of higher order mean curvatures. 
	
	\begin{theorem}\label{thm5-H}
		Let $\psi:\Sigma^{n}\rightarrow\mathcal{H}^{n+1}$ be a complete spacelike hypersurface of $\mathcal{H}^{n+1}$ contained in the closure of the interior domain enclosed by a spacelike hyperplane $\mathcal E_\tau$ orthogonal to a nonzero null vector $a \in \mathbb{L}^{n+2}$, and locally tangent from above to a spacelike hyperplane $\mathcal E_{\tilde{\tau}}$, with $\tilde{\tau} \leq \tau$. Suppose that $l_{a}=\lambda f_{a}$ for some positive constant $\lambda\in\mathbb{R}$, and that, for some $1\leq r\leq n-2$, the $r$-th mean curvature $H_r$ of $\Sigma^{n}$ is bounded and such that
		\begin{equation}\label{Hip-teo4}
			\beta\leq H_{r}\leq H_{r+2},
		\end{equation}
		where $\beta$ is a positive constant. Then, $\Sigma^n$ must be the spacelike hyperplane $\mathcal E_{\tilde{\tau}}$.
	\end{theorem}
	
	\begin{proof}
		Since $l_{a}=\lambda f_{a}$, we observe that
		\begin{equation}\label{eqaux:6A}
			|\nabla l_{a}|^{2}=f_{a}^{2}-l_{a}^{2}=(\lambda^{-2}-1)l_{a}^{2}.
		\end{equation}
		In particular, from \eqref{eqaux:6A} we have that $\lambda^{-2}-1\geq0$. Now, we define on $\Sigma^n$ the following operator $\mathcal{L} : C^{\infty}(\Sigma) \longrightarrow C^{\infty}(\Sigma)$ by
		\begin{equation}\label{Operador L}
			\mathcal{L} f = \displaystyle\frac{\lambda^{-1}}{c_{r+1}}L_{r+1} f +\frac{1}{c_{r}}L_{r} f,
		\end{equation}
		where $c_i=(i+1)\displaystyle{{n\choose i+1}}$. From equation \eqref{Lrla} and the hypothesis \eqref{Hip-teo4} , we obtain that
		\begin{equation}\label{lambda-equation}
			\begin{aligned}
				\mathcal{L} (l_a)
				& = \dfrac{\lambda^{-1}}{c_{r+1}} L_{r+1} l_a + \dfrac{1}{c_r} L_r l_a \\
				& = \dfrac{\lambda^{-1}}{c_{r+1}} \{ c_{r+1} ( -l_a h_{r+1} + f_a H_{r+2}  )  \} + \dfrac{1}{c_r} \{ c_r( -l_a H_r + f_a H_{r+1}  )  \} \\
				& = \lambda^{-2} H_{r+2} l_a - l_a H_r \geq (\lambda^{-2} - 1) H_r l_a \geq 0
			\end{aligned}
		\end{equation}
		
		On the other hand, we can reason as in the beginning of the proof of Theorem~\ref{thm3-H} to guarantee the existence of an elliptic point on $\Sigma^n$. Hence, since we are also supposing that $H_r > 0$, it follows from Lemma \ref{lemma:ellipticity of Lr} that $P_j$ is positive definite, consequently, $H_j$ is positive for all $1 \leq j \leq r-1$. Moreover, taking into account once more that $H_2 > 0$ and $H_2 < H^2$, from \eqref{norm of the shape operator} we get once more that
		$$\sum_{i} \lambda_i^2 \leq n^2H^2.$$
		Consequently, the boundedness of $H$ implies in the boundedness of all principal curvatures of $\Sigma^n$ and, hence, ${\rm tr}(P_j) = c_j H_j$ is also bounded. 
		From Gauss equation (\ref{Gauss equation}) we also have that $K_\Sigma$ is bounded from bellow. So, since our assumption that $\Sigma^n$ is contained in the closure of the interior domain enclosed by $\mathcal E_{\tau}$ implies that $l_a$ is bounded, we can apply Lemma~\ref{lemma:Maximum Principle} to get a sequence $(p_k)_{k\geq 1}$ in $\Sigma^n$ such that
		\begin{equation}\label{teo4-omori}
			\lim_kl_a(p_k)=\sup_{\Sigma}l_a,\,\,\,\,\,\lim_k|\nabla l_a(p_k)|=0\,\,\,\,\,{\rm and}\,\,\,\,\limsup_k\mathcal L(\,l_a\,)(p_k)\leq0.
		\end{equation}
		Hence, from \eqref{lambda-equation} and \eqref{teo4-omori} we obtain that
		\begin{equation}\label{teo4-final}
			0\geq\limsup_k \mathcal{L}(\,l_a\,)(p_k)\geq (\lambda^2 - 1) \beta (\sup_{\Sigma}l_a)\geq0.
		\end{equation}
		Thus, since $\sup_\Sigma l_a > 0$, we get $\lambda^{-2}-1=0$. Consequently, returning to \eqref{lambda-equation} we conclude that $\lambda = 1$. Therefore, from \eqref{eqaux:6A}, we have that $ \nabla l_a = 0$ and the height function is constant on $\Sigma^n$, which implies that $\Sigma^n$ must be the spacelike hyperplane $\mathcal E_{\tilde{\tau}}$.
	\end{proof}
	
	\section*{Acknowledgements}
	The first author is partially supported by CAPES, Brazil. The second author is partially supported by CNPq, Brazil, grant 301970/2019-0. The third author is partially supported by CNPq, Brazil, grant 311224/2018-0.
	
	\bibliographystyle{amsplain}
	
\end{document}